\documentclass{amsart}

\newtheorem{theorem}{Theorem}[section]
\newtheorem{proposition}{Proposition}[section]

\theoremstyle{definition}
\newtheorem{definition}[theorem]{Definition}
\newtheorem{example}[theorem]{Example}

\theoremstyle{remark}
\newtheorem{remark}[theorem]{Remark}

\numberwithin{equation}{section}

\begin{document}

\title{Jordan left $\{\lowercase{g}, \lowercase{h}\}$-derivations over some algebras}


\author{Arindam Ghosh}
\address{Department of Mathematics, Indian Institute of Technology Patna, Patna-801106}
\curraddr{}
\email{E-mail: arindam.pma14@iitp.ac.in}
\thanks{}

\author{Om Prakash$^{\star}$}
\address{Department of Mathematics, Indian Institute of Technology Patna, Patna-801106}
\curraddr{}
\email{om@iitp.ac.in}
\thanks{* Corresponding author}



\date{}

\dedicatory{}

\begin{abstract}
In this article, left $\{g, h\}$-derivation and Jordan left $\{g, h\}$-derivation on algebras are introduced. It is shown that there is no Jordan left $\{g, h\}$-derivation over $\mathcal{M}_n(C)$ and  $\mathbb{H}_{\mathbb{R}}$, for $g\neq h$. Examples are given which show that every Jordan left $\{g, h\}$-derivation over $\mathcal{T}_n(C)$, $\mathcal{M}_n(C)$ and $\mathbb{H}_{\mathbb{R}}$ are not left $\{g, h\}$-derivations. Moreover, we characterize left $\{g, h\}$-derivation and Jordan left $\{g, h\}$-derivation over $\mathcal{T}_n(C)$, $\mathcal{M}_n(C)$ and $\mathbb{H}_{\mathbb{R}}$ respectively. Also, we prove the result of Jordan left $\{g, h\}$-derivation to be a left $\{g, h\}$-derivation over tensor products of algebras as well as for algebra of polynomials.
\end{abstract}

\maketitle

\section{Introduction}

Throughout this paper, $C$ represents a $2$-torsion free commutative ring with unity unless otherwise stated. A ring $R$ is a $2$-torsion free if $2a=0$ for $a\in R$ implies $a=0$. Let $R$ be a ring. An additive map $D:R \rightarrow R$ is said to be a derivation if $D(ab)=D(a)b+aD(b)$, for all $a,b\in R$ and a Jordan derivation if $D(a^2)=D(a)a+aD(a)$, for all $a\in R$. Jordan derivation over rings and algebras have been studied by many researchers \cite{her,cus,ben,bre,bresa,gho,zha,zhan}. In 1990, Jordan left derivation was introduced by Bre\v{s}ar and Vukman  \cite{bresar}. They proved that the existence of a nonzero left derivation of a prime ring of characteristic not 2 implies the commutativity of the ring. After that many new results have been established on Jordan left derivations over different rings and algebras \cite{ash,bres,den,ghos,vuk,xu}. Recently, in 2016, Bre\v{s}ar introduced $\{g, h\}$-derivation and studied over semiprime algebras and tensor product of algebras \cite{bresar}. Let $A$ be an algebra over $C$ and $f,g,h:A \rightarrow A$ be linear maps. Then $f$ is said to be a left derivation if $f(ab)=af(b)+bf(a)$ and $f$ is said to be a $\{g, h\}$-derivation if $f(ab)=g(a)b+ah(b)=h(a)b+ag(b)$.

A linear map $f:A \rightarrow A$ is said to be a left centralizer if $f(ab)=f(a)b$, and it is a right centralizer if $f(ab)=af(b)$, for all $a,b\in A$. It is a two sided centralizer if $f$ is both left as well as right centralizers. Note that if $A$ has an identity element, then $f$ is a left centralizer iff there exist an element $\alpha \in A$ such that $f(a)=\alpha a$, for all $a\in A$. Also, $f$ is a right centralizer iff there exist an element $\beta \in A$ such that $f(a)=a \beta $, for all $a\in A$.

Motivated by left derivation \cite{bres} and $\{g, h\}$-derivation \cite{bresar}, we introduce left $\{g, h\}$-derivation over $A$ as follows:
\begin{definition} The map $f$ is a left $\{g, h\}$-derivation if
\begin{equation}
\label{1}
f(ab)=ag(b)+bh(a)=ah(b)+bg(a), ~\text{for all}~ a, b \in A.
\end{equation}
\end{definition}
Clearly, if $f=g=h$, then $f$ is a left derivation. Similarly, we define Jordan left $\{g, h\}$-derivation.
\begin{definition} The map $f$ is said to be a Jordan left $\{g, h\}$-derivation if
\begin{equation}
\label{2}
f(a\circ b)=2(ag(b)+bh(a)), ~\text{for all}~ a, b \in A
\end{equation} (where, $a\circ b=ab+ba$).
\end{definition}
Since $a\circ b=b\circ a$, whenever $f$ is a Jordan left $\{g, h\}$-derivation on $A$, then $f(a\circ b)=2(bg(a)+ah(b))$ and $f(a\circ b)=ag(b)+bh(a)+bg(a)+ah(b)$, for all $a,b \in A$. Also, $f(a^2)=a(g(a)+h(a))$, for all $a\in A$.

Note that, if $A$ is a 2-torsion free commutative algebra over $C$, then every Jordan left $\{g, h\}$-derivation is a left $\{g, h\}$-derivation. This is not true when $A$ is not $2$-torsion free. For example:

\begin{example}
Let $A=\mathbb{Z}_4$. Then $A$ is a commutative $\mathbb{Z}$-algebra, which is not $2$-torsion free. Now $f:A\rightarrow A$ is defined as, $f(x)=2x$, for all $x\in A$. Then $f$ is a Jordan left $\{f, f\}$-derivation. But $f([1][1])=[2]\neq [4]=[1]f([1])+[1]f([1])$. Therefore, $f$ is not a left $\{f, f\}$-derivation.
\end{example}

Now, let $f$ be a left $\{g,h\}$ derivation over $A$ and $\lambda=g(1)+h(1)$. Now define $d$ on $A$ by $d(a)=f(a)-\lambda a$. If $\lambda\in Z$ (where $Z$ is the center of $A$), then $d$ is a left derivation. So, every left $\{g,h\}$ derivation can be written as
\begin{equation}
\label{5}
f(a)=\lambda a+d(a),  ~\text{for all}~ a\in A.
\end{equation}

Also, if $f$ is of the form \eqref{5}, where $\lambda\in Z$ and $d$ is a left derivation on $A$, then $f$ is a left $\{g,h\}$ derivation for some linear maps $g$ and $h$, specially here, for $g=f$ and $h=d$.

\begin{remark}
If $f$ is a Jordan left $\{g, h\}$-derivation over $A$, then $f$ is a left $\{g+h, g+h\}$-derivation.
\end{remark}

\section{Jordan Left $\{\lowercase{g}, \lowercase{h}\}$-Derivation on Matrix Algebras}
Here, we start with an example of a left $\{g, h\}$-derivation which is not a $\{g, h\}$-derivation.
\begin{example}
Let $A=\mathcal{T}_2(C)$, algebra of $2\times 2$ upper triangular matrices over $C$. Also, let $e_{ij}$ be the matrix whose $(i,j)$th entry is $1$, otherwise $0$. Define $g:A\rightarrow A$ as $g(x)=ax$, for all $x\in A$ where $a=e_{11}$. Then $\textbf{0}$ is a left $\{g, -g\}$-derivation. But $g(e_{12})e_{22}+e_{12}(-g)(e_{22})=e_{12}\neq 0$, therefore, $\bold{0}$ is not $\{g, -g\}$-derivation.
\end{example}

Now, we give an example of a $\{g, h\}$-derivation which is not a left $\{g, h\}$-derivation.

\begin{example}
Define $f,g:\mathcal{T}_2(C)\rightarrow \mathcal{T}_2(C)$ as $f(x)=x+g(x)$ and $g(x)=ax-xa$ for all $x\in \mathcal{T}_2(C)$ where $a=e_{11}+e_{12}+e_{22}$. Then $f$ is an $\{f, g\}$-derivation. But $e_{11}f(e_{22})+e_{22}g(e_{11})=a\neq 0=f(e_{11}e_{22})$, therefore, $f$ is not a left $\{f, g\}$-derivation.
\end{example}

Note that every left $\{g, h\}$-derivation over $A$ is a Jordan left $\{g, h\}$-derivation but the converse need not be true.
\begin{example}
\label{ex1}
Let $A=\mathcal{T}_2(C)$, $X=\begin{bmatrix}
x_1 & x_2 \\
0 & x_3 \end{bmatrix}\in A$ and $f,g,h:A\rightarrow A$ are defined by $f(X)= \begin{bmatrix}
5x_1 & 7x_1+6x_2 \\
0 & 6x_3 \end{bmatrix}$, $g(X)= \begin{bmatrix}
x_1 & 2x_1+3x_2 \\
0 & 3x_3 \end{bmatrix}$ and $h(X)=\begin{bmatrix}
4x_1 & 5x_1+3x_2 \\
0 & 3x_3 \end{bmatrix}$ respectively. Then $f$ is a Jordan left $\{g, h\}$-derivation. Now, $f(e_{12}e_{11})=f(0)=0$, but $e_{12}g(e_{11})+e_{11}h(e_{12})=3e_{12}\neq 0$. So, $f$ is not a left $\{g, h\}$-derivation.
\end{example}

One can see easily in above example, $f(e_{12}\circ e_{11})=f(e_{12})=6e_{12} \neq 7e_{12}=g(e_{12})\circ e_{11} + e_{12}\circ h(e_{11})$. Therefore, $f$ is not a Jordan $\{g, h\}$-derivation. Also, we give an example of a nonzero left $\{g, h\}$-derivation over an algebra which is not a $\{g, h\}$-derivation.

\begin{example}
Let $A=\mathcal{T}_2(C)$, $X= \begin{bmatrix}
x_1 & x_2 \\
0 & x_3 \end{bmatrix}\in A$ and $f,g,h:A\rightarrow A$ are defined by $f(X)= \begin{bmatrix}
x_1 & x_1 \\
0 & 0 \end{bmatrix}$, $g(X)= \begin{bmatrix}
x_1 & 0 \\
0 & 0 \end{bmatrix}$ and $h(X)= \begin{bmatrix}
0 & x_1 \\
0 & 0 \end{bmatrix}$ respectively. Then $f$ is a left $\{g, h\}$-derivation. Now, $f(e_{11}(e_{11}+e_{12}))=f(e_{11}+e_{12})=e_{11}+e_{12}\neq e_{11}+2e_{12}=g(e_{11})(e_{11}+e_{12})+e_{11}h(e_{11}+e_{12}))$. Therefore, $f$ is not a $\{g, h\}$-derivation.
\end{example}

In next results, we characterize Jordan left $\{g, h\}$-derivation and left $\{g, h\}$-derivation over $\mathcal{T}_n(C)$.

\begin{theorem}
Let $\mathcal{T}_n(C)$, $n\geq 2$, be the algebra of $n\times n$ upper triangular matrices over $C$. Then $f$ is a Jordan left $\{g, h\}$-derivation over $\mathcal{T}_n(C)$ if and only if there exists $\dfrac{n(n+3)}{2}$ elements in $C$ such that
\begin{align*}
& g(A)=\sum\limits_{1 \leq i \leq j \leq n}(\sum_{k=i}^{j}a_{ik}g_{kj}^{(kk)})e_{ij}; \\
& h(A)=a_{11}h_{11}^{(11)}e_{11}
+\sum_{j=2}^{n}(a_{11}h_{1j}^{(11)}
+\sum_{k=2}^{j}a_{1k}g_{kj}^{(kk)})e_{1j}
+\sum\limits_{1 < i \leq j \leq n}(\sum_{k=i}^{j}a_{ik}g_{kj}^{(kk)})e_{ij}; \\
& f(A)=(g+h)(A),~\text{for all}~A\in \mathcal{T}_n(C)\\ & \text{where}~g_{lm}^{(np)},h_{lm}^{(np)}\in C~\text{and}~ A=\sum\limits_{1 \leq i \leq j \leq n}a_{ij}e_{ij}.
\end{align*}

Moreover, if $f$ is a Jordan left $\{g, h\}$-derivation over $\mathcal{T}_n(C)$, then $f$, $g$ and $h$ are right centralizers.
\end{theorem}

\begin{proof}
Let $f$ be a Jordan left $\{g, h\}$-derivation over $\mathcal{T}_n(C)$. Now, let
\begin{equation}
\label{n4}
g(e_{ij})=\sum\limits_{1 \leq m \leq p \leq n}g_{mp}^{(ij)}e_{mp},
\end{equation}

\begin{equation}
\label{n5}
h(e_{ij})=\sum\limits_{1 \leq m \leq p \leq n}h_{mp}^{(ij)}e_{mp},
\end{equation}
and
\begin{equation}
\label{n6}
f(e_{ij})=\sum\limits_{1 \leq m \leq p \leq n}f_{mp}^{(ij)}e_{mp}~, ~\text{where} ~g_{mp}^{(ij)},h_{mp}^{(ij)},f_{mp}^{(ij)} \in C.
\end{equation}

Since $e_{ii}^2=e_{ii}$ for $i\in \{1,2,\dots,n\}$ and $f$ is a Jordan left $\{g, h\}$-derivation over $\mathcal{T}_n(C)$, so $f(e_{ii})=e_{ii}g(e_{ii})+e_{ii}h(e_{ii})$ and this implies
\begin{equation}
\label{n7}
\begin{aligned}
& f_{ii}^{(ii)}=g_{ii}^{(ii)}+h_{ii}^{(ii)},f_{i,i+1}^{(ii)}=g_{i,i+1}^{(ii)}+h_{i,i+1}^{(ii)},\dots,f_{in}^{(ii)}=g_{in}^{(ii)}+h_{in}^{(ii)}, \\
&\text{and other entries of}~ f(e_{ii})~\text{are zero except}~\text{$(ii)$-th to}~\text{$(in)$-th entries}\\
&~(\text{by}~\eqref{n4},\eqref{n5} ~\text{and}~\eqref{n6}).
\end{aligned}
\end{equation}

Let $i\neq j$. Since $e_{ii}\circ e_{jj}=0$, by using \eqref{n4}, \eqref{n5}, \eqref{n6} and \eqref{2}, we have
\begin{equation}
\label{n8}
\begin{aligned}
g_{ii}^{(jj)}=g_{i,i+1}^{(jj)}=\dots=g_{in}^{(jj)}=0=h_{jj}^{(ii)}=h_{j,j+1}^{(ii)}=\dots=h_{jn}^{(ii)}.
\end{aligned}
\end{equation}

Let $i<j$. Since $e_{ij}\circ e_{ii}=e_{ij}=e_{ii}\circ e_{ij}$, we have
\begin{equation}
\label{n9}
\begin{aligned}
& 2h_{ik}^{(ij)}=f_{ik}^{(ij)}=2g_{ik}^{(ij)}, ~\text{for}~k=\{i,i+1,\dots,j-1\}.\\
& \text{Also,}~ 2(g_{jl}^{(ii)}+h_{il}^{(ij)})=f_{il}^{(ij)}=2(h_{jl}^{(ii)}+g_{il}^{(ij)}),~\text{for}~l=\{j,j+1,\dots,n\}\\
&\implies 2h_{il}^{(ij)}=f_{il}^{(ij)}=2g_{il}^{(ij)},~\text{for}~l=\{j,j+1,\dots,n\} ~(\text{by}~\eqref{n8})\\
&\text{and other entries of}~ f(e_{ij})~\text{are zero except}~\text{$(ii)$-th to}~\text{$(in)$-th entry.}
\end{aligned}
\end{equation}

Similarly, from $e_{jj}\circ e_{ij}=e_{ij}=e_{ij}\circ e_{jj}$,
\begin{equation}
\label{n10}
\begin{aligned}
& 2h_{jk}^{(jj)}=f_{ik}^{(ij)}=2g_{jk}^{(jj)},~\text{for}~k=\{j,j+1,\dots,n\}~(\text{by}~ \eqref{n9}),\\
& g_{jl}^{(ij)}=0=h_{jl}^{(ij)},~\text{for}~l=\{j,j+1,\dots,n\}~(\text{by}~ \eqref{n9})\\
&\text{and other entries of}~ f(e_{ij})~\text{are zero except}~\text{$(ij)$-th to}~\text{$(in)$-th entry.}
\end{aligned}
\end{equation}

Therefore, for any $i<j$, by \eqref{n9} and \eqref{n10},
\begin{equation}
\label{n11}
\begin{aligned}
& g_{ik}^{(ij)}=h_{ik}^{(ij)}=h_{jk}^{(jj)}=g_{jk}^{(jj)},~\text{for}~k=\{j,j+1,\dots,n\},\\
& g_{il}^{(ij)}=0=h_{il}^{(ij)},~\text{for}~l=\{i,i+1,\dots,j-1\}.
\end{aligned}
\end{equation}

Let $k\neq i,j$. Since $e_{kk}\circ e_{ij}=0$,
\begin{equation}
\label{n12}
\begin{aligned}
g_{kl}^{(ij)}=0=h_{kl}^{(ij)},~\text{for}~l=\{k,k+1,\dots,n\}.
\end{aligned}
\end{equation}

Hence, from \eqref{n7}-\eqref{n8} and \eqref{n10}-\eqref{n12}, all entries of $g(e_{ij})$, $h(e_{ij})$ and $f(e_{ij})$ are zero except $(ij)$-th to $(in)$-th entry, for all $1\leq i\leq j \leq n$.

Now, let $A\in\mathcal{T}_n(C)$. Then $A=\sum\limits_{1 \leq i \leq j \leq n}a_{ij}e_{ij}$, where $a_{ij}\in C$. Since $g,h$ and $f$ are linear,
\begin{equation}
\label{n13}
\begin{aligned}
& g(A)=\sum\limits_{1 \leq i \leq j \leq n}(\sum_{k=i}^{j}a_{ik}g_{kj}^{(kk)})e_{ij}, \\
& h(A)=a_{11}h_{11}^{(11)}e_{11}
+\sum_{j=2}^{n}(a_{11}h_{1j}^{(11)}
+\sum_{k=2}^{j}a_{1k}g_{kj}^{(kk)})e_{1j}
+\sum\limits_{1 < i \leq j \leq n}(\sum_{k=i}^{j}a_{ik}g_{kj}^{(kk)})e_{ij}, \\
& f(A)=(g+h)(A),~(\text{by}~\eqref{n7},\eqref{n10}~\text{and}~\eqref{n11}).
\end{aligned}
\end{equation}

Therefore, the number of elements of $C$ requires to express $g,h$ and $f$ is equal to $n+n+(n-1)+\dots+1=\frac{n(n+3)}{2}$.

Conversely, let $g,h$ and $f$ be of the form \eqref{n13}, for all $A\in\mathcal{T}_n(C)$, where $A=\sum\limits_{1 \leq i \leq j \leq n}a_{ij}e_{ij}$ and $a_{ij}\in C$. Then $g(A)=A\alpha,~h(A)=A\alpha^{\prime}$ and $f(A)=A(\alpha+\alpha^{\prime})$ where $\alpha=\sum\limits_{1 \leq i \leq j \leq n}g_{ij}^{(ii)}e_{ij}$ and $\alpha^{\prime}=\sum_{j=1}^{n}h_{1j}^{(11)}e_{1j}+\sum\limits_{1 < i \leq j \leq n}g_{ij}^{(ii)}e_{ij}$. Now, let $B\in\mathcal{T}_n(C)$. Then $B=\sum\limits_{1 \leq i \leq j \leq n}b_{ij}e_{ij}$, where $b_{ij}\in C$ and by direct computation,
\begin{align*}
f(AB+BA)=2(Ag(B)+Bh(A)).
\end{align*}

The last conclusion can be easily seen from the converse part.

\end{proof}

\begin{theorem}
Let $\mathcal{T}_n(C)$, $n\geq 2$, be the algebra of $n\times n$ upper triangular matrices over $C$. Then $f$ is a left $\{g, h\}$-derivation over $\mathcal{T}_n(C)$ if and only if there exists $2n$ elements in $C$ such that
\begin{align*}
& g(A)=\sum_{i=1}^{n}a_{11}g_{1i}^{(11)}e_{1i};~
 h(A)=\sum_{i=1}^{n}a_{11}h_{1i}^{(11)}e_{1i};\\
& f(A)=(g+h)(A),~\text{for all}~A\in \mathcal{T}_n(C),\\ &\text{where}~g_{lm}^{(np)},h_{lm}^{(np)}\in C~\text{and}~ A=\sum\limits_{1 \leq i \leq j \leq n}a_{ij}e_{ij}.
\end{align*}
\end{theorem}

\begin{proof}
Let $f$ be a left $\{g, h\}$-derivation over $\mathcal{T}_n(C)$ where $g$, $h$ and $f$ be of the form \eqref{n4}, \eqref{n5} and \eqref{n6} respectively. Since every left $\{g, h\}$-derivation over $\mathcal{T}_n(C)$ is a Jordan left $\{g, h\}$-derivation, so \eqref{n7} and \eqref{n8} hold.

Let $i<j$. Since $e_{ij}e_{ii}=0$, from $e_{ij}g(e_{ii})+e_{ii}h(e_{ij})=0=e_{ij}h(e_{ii})+e_{ii}g(e_{ij})$ (by \eqref{1}), we have
\begin{equation}
\label{n14}
\begin{aligned}
& h_{ik}^{(ij)}=0=g_{ik}^{(ij)}, ~\text{for}~k=\{i,i+1,\dots,j-1\}\\
& \text{and}~ g_{jl}^{(ii)}+h_{il}^{(ij)}=0=h_{jl}^{(ii)}+g_{il}^{(ij)},~\text{for}~l=\{j,j+1,\dots,n\}\\
&\implies h_{il}^{(ij)}=0=g_{il}^{(ij)},~\text{for}~l=\{j,j+1,\dots,n\} ~(\text{by}~\eqref{n8}).
\end{aligned}
\end{equation}

Similarly, from the identity $e_{jj}e_{ij}=0$,
\begin{equation}
\label{n15}
\begin{aligned}
& h_{jk}^{(jj)}=0=g_{jk}^{(jj)},~\text{for}~k=\{j,j+1,\dots,n\}\\
& g_{jl}^{(ij)}=0=h_{jl}^{(ij)},~\text{for}~l=\{j,j+1,\dots,n\}.
\end{aligned}
\end{equation}

By \eqref{n7} and \eqref{n15},
\begin{equation}
\label{n16}
\begin{aligned}
f(e_{jj})=0,~\text{for all}~ j>1.
\end{aligned}
\end{equation}

By \eqref{n8} and \eqref{n15},
\begin{equation}
\label{n17}
\begin{aligned}
g(e_{jj})=h(e_{jj})=0,~\text{for all}~ j>1.
\end{aligned}
\end{equation}

Let $k\neq i,j$. Since $e_{kk}e_{ij}=0$,
\begin{equation}
\label{n18}
\begin{aligned}
g_{kk}^{(ij)}=g_{k,k+1}^{(ij)}=\dots=g_{kn}^{(ij)}=0=h_{kk}^{(ij)}=h_{k,k+1}^{(ij)}=\dots=h_{kn}^{(ij)}.
\end{aligned}
\end{equation}

By \eqref{n14}, \eqref{n15} and \eqref{n18},
\begin{equation}
\label{n19}
\begin{aligned}
g(e_{ij})=h(e_{ij})=0,~\text{for all}~ i<j.
\end{aligned}
\end{equation}

Now,
\begin{equation}
\label{n20}
\begin{aligned}
f(e_{ij})&=e_{ii}g(e_{ij})+e_{ij}h(e_{ii})=0,~\text{for all}~ i<j~(\text{by}~ \eqref{n19},\eqref{n5},\eqref{n8}).
\end{aligned}
\end{equation}

Let $A\in\mathcal{T}_n(C)$. Then $A=\sum\limits_{1 \leq i \leq j \leq n}a_{ij}e_{ij}$, where $a_{ij}\in C$. Since $g,h$ and $f$ are linear, by \eqref{n4}-\eqref{n8},\eqref{n16},\eqref{n17},\eqref{n19} and \eqref{n20}, we have
\begin{align*}
& g(A)=\sum_{i=1}^{n}a_{11}g_{1i}^{(11)}e_{1i},~
 h(A)=\sum_{i=1}^{n}a_{11}h_{1i}^{(11)}e_{1i},\\
& f(A)=(g+h)(A),~\text{for all}~A\in \mathcal{T}_n(C).
\end{align*}

Thus, the number of elements from $C$ requires to express $g,h$ and $f$ is $n+n=2n$.

Converse can be proved easily by direct computation.

\end{proof}

Now, we give an example of a Jordan left $\{g, h\}$-derivation $f$ on $\mathcal{T}_n(C)$ where $g,h$ and $f$ are not left centralizers.
\begin{example}
Consider $A= \begin{bmatrix}
a_{11} & a_{12} \\
0 & a_{22} \end{bmatrix}\in \mathcal{T}_2(C)$ and $g,h,f:\mathcal{T}_2(C)\rightarrow \mathcal{T}_2(C)$ are defined by $g(A)= \begin{bmatrix}
a_{11} & -a_{12} \\
0 & -a_{22} \end{bmatrix}$, $h(A)= \begin{bmatrix}
-a_{11} & a_{11}-a_{12} \\
0 & -a_{22} \end{bmatrix}$ and $f(A)= \begin{bmatrix}
0 & a_{11}-2a_{12} \\
0 & -2a_{22} \end{bmatrix}$ respectively. Then $f$ is a Jordan left $\{g, h\}$-derivation on $\mathcal{T}_2(C)$. Here, if $A=\begin{bmatrix}
1 & 2 \\
0 & 3 \end{bmatrix}$ and $B=\begin{bmatrix}
4 & 5 \\
0 & 6 \end{bmatrix}$, then $g(AB)\neq g(A)B$, $h(AB)\neq h(A)B$ and $f(AB)\neq f(A)B$. Therefore, $g,h$ and $f$ are not left centralizers.
\end{example}

Our next result characterize Jordan left $\{g, h\}$-derivation on full matrix algebras $\mathcal{M}_n(C)$.

\begin{theorem}
Let $f$ be a Jordan left $\{g, h\}$-derivation over $\mathcal{M}_n(C)$, $n\geq 2$, the algebra of $n\times n$ matrices over $C$. Then $g=h$.
\end{theorem}

\begin{proof}
Suppose $f$ is a Jordan left $\{g, h\}$-derivation over $\mathcal{M}_n(C)$. Now, let
\begin{equation}
\label{6}
g(e_{ij})=\sum_{k=1}^{n}\sum_{l=1}^{n}g_{kl}^{(ij)}e_{kl}
\end{equation}

and \begin{equation}
\label{7}
h(e_{ij})=\sum_{k=1}^{n}\sum_{l=1}^{n}h_{kl}^{(ij)}e_{kl}~, ~\text{where} ~g_{kl}^{(ij)},h_{kl}^{(ij)} \in C.
\end{equation}

Let $i\neq j$. Since $C$ is $2$-torsion free,
\begin{equation}
\label{8}
2(g(e_{ii})e_{jj}+e_{jj}h(e_{ii}))=f(e_{ii}\circ e_{jj})=0 \implies g(e_{ii})e_{jj}+e_{jj}h(e_{ii})=0.
\end{equation}

Comparing the coefficients of $e_{i1},\dots,e_{in},e_{j1},\dots,e_{jn}$ from \eqref{8},

\begin{equation}
\label{9}
g_{i1}^{(jj)}=\dots=g_{in}^{(jj)}=0=h_{j1}^{(ii)}=\dots=h_{jn}^{(ii)}.
\end{equation}

Similarly, from $h(e_{ii})e_{jj}+e_{jj}g(e_{ii})=0$,
\begin{equation}
\label{10}
h_{i1}^{(jj)}=\dots=h_{in}^{(jj)}=0=g_{j1}^{(ii)}=\dots=g_{jn}^{(ii)}.
\end{equation}

Now, $e_{ij}=e_{ii}\circ e_{ij}$. Using \eqref{9} and \eqref{10},
\begin{equation}
\label{11}
f(e_{ij})=2(g_{i1}^{(ij)}e_{i1}+\dots+g_{in}^{(ij)}e_{in})=2(h_{i1}^{(ij)}e_{i1}+\dots+h_{in}^{(ij)}e_{in}).
\end{equation}

In a similar way, we have
\begin{equation}
\label{12}
\begin{aligned}
f(e_{ij})&=2(g_{j1}^{(jj)}e_{i1}+\dots+g_{jn}^{(jj)}e_{in}+h_{j1}^{(ij)}e_{j1}+\dots+h_{jn}^{(ij)}e_{jn})\\
&=2(h_{j1}^{(jj)}e_{i1}+\dots+g_{jn}^{(jj)}e_{in}+h_{j1}^{(ij)}e_{j1}+\dots+h_{jn}^{(ij)}e_{jn}).
\end{aligned}
\end{equation}

From \eqref{12},
\begin{equation}
\label{13}
g_{j1}^{(jj)}=h_{j1}^{(jj)},\dots,g_{jn}^{(jj)}=h_{jn}^{(jj)}, ~\text{for all}~ j=1,\dots,n.
\end{equation}

From \eqref{9},\eqref{10} and \eqref{13},
\begin{equation}
\tag{A1} \label{A1}
g(e_{jj})=h(e_{jj})~\text{for all}~ j=1,\dots,n.
\end{equation}

From \eqref{11},
\begin{equation}
\label{14}
g_{i1}^{(ij)}=h_{i1}^{(ij)},\dots,g_{in}^{(ij)}=h_{in}^{(ij)}~\text{for all}~ i\neq j.
\end{equation}

From \eqref{12},
\begin{equation}
\label{15}
g_{j1}^{(ij)}=h_{j1}^{(ij)},\dots,g_{jn}^{(ij)}=h_{jn}^{(ij)}~\text{for all}~ i\neq j.
\end{equation}

Now, let $k\neq i$ and $k\neq j$. From $e_{kk}\circ e_{ij}=0$,
\begin{equation}
\label{16}
g_{k1}^{(ij)}=\dots=g_{kn}^{(ij)}=0=h_{k1}^{(ij)}=\dots=h_{kn}^{(ij)}.
\end{equation}

From \eqref{14}, \eqref{15} and \eqref{16},
\begin{equation}
\tag{A2} \label{A2}
g(e_{ij})=h(e_{ij})~\text{for all}~ i\neq j.
\end{equation}

Since $g$ and $h$ are linear maps, using \eqref{A1} and \eqref{A2}, $g(X)=h(X)$, for all $X\in \mathcal{M}_n(C)$.
\end{proof}

Now, we give an example of a Jordan left $\{g, g\}$-derivation over $\mathcal{M}_2(C)$ which is not a left  $\{g, g\}$-derivation.
\begin{example}
\label{ex4}
Let $A=\mathcal{M}_2(C)$ and $X= \begin{bmatrix}
x_1 & x_2 \\
x_3 & x_4 \end{bmatrix}\in A$. We consider $f,g:A\rightarrow A$ defined as\\ $f(X)= \begin{bmatrix}
2x_1+6x_2 & 4x_1+8x_2 \\
2x_3+6x_4 & 4x_3+8x_4 \end{bmatrix}$ and $g(X)= \begin{bmatrix}
x_1+3x_2 & 2x_1+4x_2 \\
x_3+3x_4 & 2x_3+4x_4 \end{bmatrix}$ respectively. Then $f$ is a Jordan left $\{g, g\}$-derivation. Now, $f(e_{12}e_{11})=f(0)=0$, but $e_{12}g(e_{11})+e_{11}g(e_{12})=3e_{11}+4e_{12}\neq 0$. Therefore, $f$ is not a left $\{g, g\}$-derivation.
\end{example}

\begin{remark}
It is known that if $C$ is a prime (semiprime) ring, then $\mathcal{M}_n(C)$ is a prime (semiprime) ring, for $n\geq 2$. In Example \ref{ex4}, if $C$ is a commutative prime (semiprime) ring, then $\mathcal{M}_2(C)$ becomes a prime (semiprime) algebra over $C$. Therefore, Example \ref{ex4} shows that every Jordan left $\{g, h\}$-derivation over a prime (semiprime) algebra need not be a left $\{g, h\}$-derivation.
\end{remark}

Also, we characterize Jordan left $\{g, g\}$-derivation and left $\{g, g\}$-derivation over $\mathcal{M}_n(C)$.

\begin{theorem}
Let $\mathcal{M}_n(C)$, $n\geq 2$, be the algebra of $n\times n$ full matrices over $C$. Then $f$ is a Jordan left $\{g, g\}$-derivation over $\mathcal{M}_n(C)$ if and only if there exists $n^2$ elements in $C$ such that
\begin{align*}
& g(A)=\sum_{j=1}^{n}\sum_{i=1}^{n}(\sum_{k=1}^{n}a_{ik}g_{kj}^{(kk)})e_{ij},\\
& f(A)=2g(A),~\text{for all}~A\in \mathcal{T}_n(C) ~\text{where,}~g_{lm}^{(np)}\in C,~ A=\sum_{j=1}^{n}\sum_{i=1}^{n} a_{ij}e_{ij}.
\end{align*}

Moreover, if $f$ is a Jordan left $\{g, g\}$-derivation over $\mathcal{M}_n(C)$, then $f$ and $g$ are right centralizers.
\end{theorem}

\begin{proof}
Let $f$ be a Jordan left $\{g, g\}$-derivation over $\mathcal{M}_n(C)$. Now, let $g$ be of the form \eqref{6} and
\begin{equation}
\label{n32}
f(e_{ij})=\sum_{k=1}^{n}\sum_{l=1}^{n}f_{kl}^{(ij)}e_{kl}~, ~\text{where} ~f_{kl}^{(ij)} \in C.
\end{equation}

Let $i\in \{1,2,\dots,n\}$. Since $e_{ii}^2=e_{ii}$ and $f$ is a Jordan left $\{g, g\}$-derivation over $\mathcal{M}_n(C)$,
\begin{equation}
\label{n33}
\begin{aligned}
f(e_{ii})=2\sum_{k=1}^{n}g_{ik}^{(ii)}e_{ik}.
\end{aligned}
\end{equation}

Let $i\neq j$. Then from $e_{ij}=e_{ii}\circ e_{ij}$,
\begin{equation}
\label{n34}
\begin{aligned}
f(e_{ij})=2\sum_{k=1}^{n}(g_{ik}^{(ij)}+g_{jk}^{(ii)})e_{ik}.
\end{aligned}
\end{equation}

Similarly, as $e_{ij}=e_{ij}\circ e_{jj}$,
\begin{equation}
\label{n35}
\begin{aligned}
f(e_{ij})=2(\sum_{k=1}^{n}g_{jk}^{(jj)}e_{ik}+\sum_{k=1}^{n}g_{jk}^{(ij)}e_{jk}).
\end{aligned}
\end{equation}

Now, by \eqref{n34} and \eqref{n35},
\begin{equation}
\label{n36}
\begin{aligned}
g_{jk}^{(ij)}=0, ~\text{for}~k=1,2,\dots,n.
\end{aligned}
\end{equation}

Let $k\neq i,j$. Since $0=e_{ij}\circ e_{kk}$,
\begin{equation}
\label{n37}
\begin{aligned}
g_{kl}^{(ij)}=0=g_{jl}^{(kk)},~\text{for}~l=1,2,\dots,n.
\end{aligned}
\end{equation}

Now, by \eqref{n34}, \eqref{n35} and \eqref{n37},
\begin{equation}
\label{n38}
\begin{aligned}
g_{ik}^{(ij)}=g_{jk}^{(jj)},~\text{for}~k=1,2,\dots,n.
\end{aligned}
\end{equation}
Therefore, for $i\neq j$, by using \eqref{n35}-\eqref{n38},
\begin{equation}
\label{n39}
\begin{aligned}
 f(e_{ij})=2\sum_{k=1}^{n}g_{jk}^{(jj)}e_{ik},~
 g(e_{ij})=\sum_{k=1}^{n}g_{jk}^{(jj)}e_{ik},~
 g(e_{ii})=\sum_{k=1}^{n}g_{ik}^{(ii)}e_{ik}.
\end{aligned}
\end{equation}
Let $A\in\mathcal{M}_n(C)$. Then $A=\sum_{i=1}^{n}\sum_{j=1}^{n}a_{ij}e_{ij}$, where $a_{ij}\in C$. Hence, by \eqref{n33} and \eqref{n39},
\begin{equation}
\label{n40}
\begin{aligned}
& g(A)=\sum_{j=1}^{n}\sum_{i=1}^{n}(\sum_{k=1}^{n}a_{ik}g_{kj}^{(kk)})e_{ij},\\
& f(A)=2g(A).
\end{aligned}
\end{equation}
Also, the required number of elements from $C$ to express $g$ and $f$ is $n+n+\dots+n~(n~\text{times})=n^2$.\\
Conversely, let $g$ and $f$ be of the form \eqref{n40}, for all $A\in\mathcal{M}_n(C)$, where $A=\sum_{i=1}^{n}\sum_{j=1}^{n}a_{ij}e_{ij}$ and $a_{ij}\in C$. Then $g(A)=A\alpha$ and $f(A)=2A\alpha$, where $\alpha=\sum_{i=1}^{n}\sum_{j=1}^{n}g_{ij}^{(ii)}e_{ij}$. Now, let $B\in\mathcal{M}_n(C)$. Then $B=\sum\limits_{1 \leq i \leq j \leq n}b_{ij}e_{ij}$, where $b_{ij}\in C$. Thus, by direct computation
\begin{align*}
f(AB+BA)=2(Ag(B)+Bg(A)).
\end{align*}

The last part can easily be derived from the converse part.

\end{proof}

\begin{theorem}
Let $\mathcal{M}_n(C)$, $n\geq 2$, be the algebra of $n\times n$ full matrices over $C$. Then $f$ is a left $\{g, g\}$-derivation over $\mathcal{M}_n(C)$ if and only if $f=g=0$.
\end{theorem}
\begin{proof}
Let $f$ be a left $\{g, g\}$-derivation over $\mathcal{M}_n(C)$ where $g$ and $f$ be of the form \eqref{6} and \eqref{n32} respectively. Since every left $\{g, g\}$-derivation over $\mathcal{M}_n(C)$ is a Jordan left $\{g, g\}$-derivation, therefore, \eqref{n33} and \eqref{n40} hold.

Let $i\neq j$. Since $e_{ij}=e_{ii}e_{ij}$,
\begin{equation}
\label{n41}
\begin{aligned}
f(e_{ij})=e_{ii}g(e_{ij})+e_{ij}g(e_{ii})=f(e_{ij} e_{ii})=0.
\end{aligned}
\end{equation}

Now, since $e_{ii}e_{ji}=0$,
\begin{equation}
\label{n42}
\begin{aligned}
g_{ik}^{(ii)}=0,~\text{for}~k=1,2,\dots,n.
\end{aligned}
\end{equation}

By \eqref{n33} and \eqref{n40}-\eqref{n42}, $f=g=0$.

\end{proof}
The next example is of a Jordan left $\{g, g\}$-derivation $f$ on $\mathcal{M}_n(C)$ where $g$ and $f$ are not left centralizers.
\begin{example}
Let $A= \begin{bmatrix}
a_{11} & a_{12} \\
a_{21} & a_{22} \end{bmatrix}\in \mathcal{M}_2(C)$ and $g,f:\mathcal{M}_2(C)\rightarrow \mathcal{M}_2(C)$ are defined by $g(A)= \begin{bmatrix}
a_{11}-a_{12} & -a_{11} \\
a_{21}-a_{22} & -a_{21} \end{bmatrix}$ and $f(A)= 2g(A)$ respectively. Then $f$ is a Jordan left $\{g, g\}$-derivation on $\mathcal{M}_2(C)$. In this case, for $A=\begin{bmatrix}
1 & 2 \\
3 & 4 \end{bmatrix}$ and $B= \begin{bmatrix}
5 & 6 \\
7 & 8 \end{bmatrix}$, $g(AB)\neq g(A)B$ and $f(AB)\neq f(A)B$. Therefore, $g$ and $f$ are not left centralizers.
\end{example}

\section{Jordan Left $\{\lowercase{g}, \lowercase{h}\}$-Derivation on Tensor Products of Algebras}
Now, we discuss a result on Jordan left $\{g, h\}$-derivation over tensor products of algebras. Let $A,B$ and $D$ be algebras over a field $\mathbb{F}$. Then a map $\beta:A\times B\rightarrow D$ is said to be $\mathbb{F}$-bilinear if $\beta(a_1+a_2,b)=\beta(a_1,b)+\beta(a_2,b)$, $\beta(a,b_1+b_2)=\beta(a,b_1)+\beta(a,b_2)$ and $\beta(ra,b)=\beta(a,rb)=r\beta(a,b)$, for all $a,a_1   ,a_2\in A$, $b,b_1,b_2\in B$ and $r\in \mathbb{F}$. A tensor product of two algebras $A$ and $B$ is an $\mathbb{F}$-algebra $A\otimes B$, together with an $\mathbb{F}$-bilinear map $\tau: A\times B \rightarrow A\otimes B$ such that for any $\mathbb{F}$-algebra $D$ and any $\mathbb{F}$-bilinear map $\beta:A\times B \rightarrow D$ there exists a unique $\mathbb{F}$-linear map $f:A\otimes B\rightarrow D$ such that $f\circ \tau=\beta$. For $a\in A, b\in B$ the image $\tau(a,b)$ is denoted by $a\otimes b$.

\begin{theorem}
If an algebra $A$ over a field $\mathbb{F}$ with char$(\mathbb{F})\neq 2$, has the property that every Jordan left $\{g, h\}$-derivation of $A$ is a left $\{g, h\}$-derivation, then the algebra $A\otimes S$ has the same property where $S$ is a commutative algebra over $\mathbb{F}$.
\end{theorem}

\begin{proof}
Let $f$ be a Jordan left $\{g, h\}$-derivation on $A\otimes S$. Let $\{b_t\mid t\in T, ~ T ~ an~ index~ set\}$ be a basis of $S$ and $u\in A\otimes S$. Then
\begin{equation}
\label{17}
f(u)=\sum\limits_{t\in T} f_t(u)\otimes b_t, ~g(u)=\sum\limits_{t\in T} g_t(u)\otimes b_t ~\text{and}~ h(u)=\sum\limits_{t\in T} h_t(u)\otimes b_t
\end{equation}

where $f_t(u)=g_t(u)=h_t(u)=0$ for all but finitely many $t\in T$.

Now, let $x,y\in A$ and $r,s\in S$. Since $f$ is a Jordan left $\{g, h\}$-derivation on $A\otimes S$,
\begin{equation}
\label{18}
\begin{aligned}
& f((xy+yx)\otimes rs)\\
&=f((x\otimes r)(y\otimes s)+(y\otimes s)(x\otimes r))=2[(x\otimes r)g(y\otimes s)+(y\otimes s)h(x\otimes r)],\\
&\text{which implies,} \sum\limits_{t\in T}f_t((xy+yx)\otimes rs)\otimes b_t\\
&=2[(x\otimes r)(\sum\limits_{w\in T}g_w(y\otimes s)\otimes b_w)+(y\otimes s)(\sum\limits_{w\in T}h_w(x\otimes r)\otimes b_w)]\\
&=2[(\sum\limits_{w\in T}xg_w(y\otimes s))\otimes b_wr + (\sum\limits_{w\in T}yh_w(x\otimes r))\otimes b_ws]
\end{aligned}
\end{equation}
where, $b_wr=\sum\limits_{t\in T}\alpha_{tw}b_t$, $b_ws=\sum\limits_{t\in T}\beta_{tw}b_t$ and $\alpha_{tw}, \beta_{tw}\in \mathbb{F}$.

\begin{equation}
\label{19}
\begin{aligned}
&\text{R.H.S. of \eqref{18}}=2[\sum\limits_{t\in T}x(\sum\limits_{w\in T}\alpha_{tw}g_w(y\otimes s))\otimes b_t] \\
&+ [\sum\limits_{t\in T}y(\sum\limits_{w\in T}\beta_{tw}h_w(x\otimes r))\otimes b_t]\\
&\implies f_t((xy+yx)\otimes rs)=2[x(\sum\limits_{w\in T}\alpha_{tw}g_w(y\otimes s)) + y(\sum\limits_{w\in T}\beta_{tw}h_w(x\otimes r))].
\end{aligned}
\end{equation}

Let $\widetilde{f}(x)=f_t(x\otimes rs)$, $\widetilde{g}(y)=\sum\limits_{w\in T}\alpha_{tw}g_w(y\otimes s)$ and $\widetilde{h}(x)=\sum\limits_{w\in T}\beta_{tw}h_w(x\otimes r)$, for all $x,y\in A$. Then, by \eqref{19} $\widetilde{f}$ is a Jordan left $\{\widetilde{g},\widetilde{h}\}$-derivation on $A$. So, $\widetilde{f}$ is a left $\{\widetilde{g},\widetilde{h}\}$-derivation, by assumption. So,
\begin{equation}
\label{20}
f_t(xy\otimes rs)=x(\sum\limits_{w\in T}\alpha_{tw}g_w(y\otimes s))+y(\sum\limits_{w\in T}\beta_{tw}h_w(x\otimes r)).
\end{equation}

Also,
\begin{equation}
\label{21}
\begin{aligned}
&f((x\otimes r)(y\otimes s))=f(xy\otimes rs)=\sum\limits_{t\in T}f_t(xy\otimes rs)\otimes b_t \\
&=\sum\limits_{t\in T}[x(\sum\limits_{w\in T}\alpha_{tw}g_w(y\otimes s))+y(\sum\limits_{w\in T}\beta_{tw}h_w(x\otimes r))] ~~\text{(by \eqref{20})} \\
&=x(\sum\limits_{w\in T}g_w(y\otimes s))\otimes(\sum\limits_{t\in T}\alpha_{tw}b_t)+y(\sum\limits_{w\in T}h_w(x\otimes r))\otimes(\sum\limits_{t\in T}\beta_{tw}b_t)\\
&=\sum\limits_{w\in T}xg_w(y\otimes s)\otimes b_wr + \sum\limits_{w\in T}yh_w(x\otimes r)\otimes b_ws \\
&=(x\otimes r)(\sum\limits_{w\in T}g_w(y\otimes s)\otimes b_w) + (y\otimes s)(\sum\limits_{w\in T}h_w(x\otimes r)\otimes b_w) \\
&\text{(since S is commutative)}\\
&=(x\otimes r)g(y\otimes s)+(y\otimes s)h(x\otimes r).
\end{aligned}
\end{equation}
Therefore, $f(uv)=ug(v)+vh(u)$ for all $u,v \in A\otimes S$. Similarly, $f(uv)=uh(v)+vg(u)$ for all $u,v \in A\otimes S$. Hence $f$ is a left $\{g,h\}$-derivation on $A\otimes S$.
\end{proof}

\section{Jordan Left $\{\lowercase{g}, \lowercase{h}\}$-Derivation on Algebra of polynomials}
Let $A$ be an algebra over $C$ and $A[x]$, the ring of polynomials over $A$. Then $A[x]$ becomes an algebra over $C$, where scalar multiplication is defined as $\alpha(a_0+a_1x+a_2x^2+\dots+a_rx^r)=\alpha a_0+(\alpha a_1)x+(\alpha a_2)x^2+\dots+(\alpha a_r)x^r$, $\alpha \in C$ and $a_0,a_1,a_2,\dots,a_r \in A$. First, we derive a result on left $\{g,h\}$-derivation on $A[x]$.
\begin{theorem}
\label{thm4.1}
If $f$ is left $\{g,h\}$-derivation on $A$, then $\widetilde{f}$ is a left $\{\widetilde{g},\widetilde{h}\}$-derivation, where $\widetilde{F}(\sum_{j=0}^{r}a_j x^j)=\sum_{j=0}^{r}F(a_j) x^j$, for $F=f,g,h$ and $a_j\in A$.
\end{theorem}

\begin{proof}
Let $P$, $Q\in A[x]$. Then $P=\sum_{j=0}^{r}a_j x^j$ and $Q=\sum_{k=0}^{s}b_k x^k$, where $a_j, b_k \in A$. Therefore, $PQ=\sum_{n=0}^{r+s}c_n x^n$ where $c_n=\underset{j+k=n}{\sum_{k=0}^{s}\sum_{j=0}^{r}}a_j b_k$.

By definition of $\widetilde{f}$, $\widetilde{g}$, $\widetilde{h}$ and the assumption that $f$ is a left $\{g,h\}$-derivation on $A$,
\begin{align*}
&P\widetilde{g}(Q)+Q\widetilde{h}(P)\\
&=(\sum_{j=0}^{r}a_j x^j)[\sum_{k=0}^{s}g(b_k) x^k]+(\sum_{k=0}^{s}b_k x^k)[\sum_{j=0}^{r}h(a_j) x^j]\\
&=\sum_{n=0}^{r+s}[\underset{j+k=n}{\sum_{k=0}^{s}\sum_{j=0}^{r}}a_j
g(b_k)]x^n+\sum_{n=0}^{r+s}[\underset{j+k=n}{\sum_{k=0}^{s}\sum_{j=0}^{r}}b_k h(a_j)]x^n\\
&=\sum_{n=0}^{r+s}[\underset{j+k=n}{\sum_{k=0}^{s}\sum_{j=0}^{r}}(a_j
g(b_k)+b_k h(a_j))]x^n\\
&=\sum_{n=0}^{r+s}[\underset{j+k=n}{\sum_{k=0}^{s}\sum_{j=0}^{r}}f(a_j b_k)]x^n\\
&=\widetilde{f}[\sum_{n=0}^{r+s}(\underset{j+k=n}{\sum_{k=0}^{s}\sum_{j=0}^{r}}a_j b_k)x^n]\\
&=\widetilde{f}(PQ).
\end{align*}

Similarly, $\widetilde{f}(PQ)=P\widetilde{h}(Q)+Q\widetilde{g}(P)$. Thus, $\widetilde{f}$ is a left $\{\widetilde{g},\widetilde{h}\}$-derivation on $A[x]$.
\end{proof}

Our next result characterizes Jordan left $\{g,h\}$-derivation on $A[x]$.
\begin{theorem}
\label{thm4.2}
If $f$ is Jordan left $\{g,h\}$-derivation on $A$, then $\widetilde{f}$ is a Jordan left $\{\widetilde{g},\widetilde{h}\}$-derivation, where $\widetilde{F}(\sum_{j=0}^{r}a_j x^j)=\sum_{j=0}^{r}F(a_j) x^j$, for $F=f,g,h$ and $a_j\in A$.
\end{theorem}

\begin{proof}
Let $P$, $Q\in A[x]$. Then $P=\sum_{j=0}^{r}a_j x^j$ and $Q=\sum_{k=0}^{s}b_k x^k$, where $a_j, b_k \in A$. Therefore, $PQ+QP=\sum_{n=0}^{r+s}d_n x^n$, where $d_n=\underset{j+k=n}{\sum_{k=0}^{s}\sum_{j=0}^{r}}(a_j b_k+b_k a_j)$.

By definition of $\widetilde{f}$, $\widetilde{g}$, $\widetilde{h}$ and the assumption that $f$ is a Jordan left $\{g,h\}$-derivation on $A$,
\begin{align*}
&2(P\widetilde{g}(Q)+Q\widetilde{h}(P))\\
&=\sum_{n=0}^{r+s}[\underset{j+k=n}{\sum_{k=0}^{s}\sum_{j=0}^{r}}2(a_j
g(b_k)+b_k h(a_j))]x^n\\
&=\sum_{n=0}^{r+s}[\underset{j+k=n}{\sum_{k=0}^{s}\sum_{j=0}^{r}}f(a_j b_k + b_k a_j)]x^n\\
&=\widetilde{f}(PQ+QP).
\end{align*}

Hence, $\widetilde{f}$ is a Jordan left $\{\widetilde{g},\widetilde{h}\}$-derivation on $A[x]$.
\end{proof}

\section{Jordan Left $\{\lowercase{g}, \lowercase{h}\}$-Derivation on Quaternion Algebra}
Let $\mathbb{H}_{\mathbb{R}}=\{a+bi+cj+dk\mid a,b,c,d\in \mathbb{R}, i^2=j^2=k^2=ijk=-1 \}$ be the quaternion algebra over real numbers.

\begin{proposition}
\label{thm3}
Let $f$ be a left $\{g,h\}$-derivation over $\mathbb{H}_{\mathbb{R}}$. Then the image of $f$ is the subalgebra of $\mathbb{H}_{\mathbb{R}}$ generated by $f(1)$.
\end{proposition}

\begin{proof}
Since $ij+ji=0$ and $f$ is a left $\{g,h\}$-derivation,
\begin{equation}
\label{22}
0=f(ij+ji)=2(ig(j)+jh(i)) \implies f(k)=f(ij)=ig(j)+jh(i)=0.
\end{equation}

Similarly, we get $f(i)=f(j)=0$. Now, let $q=a+bi+cj+dk \in \mathbb{H}_{\mathbb{R}}$, where $a,b,c,d\in \mathbb{R}$. Since $f$ is linear, $f(q)=af(1)$. So, $f(\mathbb{H}_{\mathbb{R}})=<f(1)>$.
\end{proof}

Now, we present the necessary and sufficient condition for Jordan left $\{g,h\}$-derivation to be a left $\{g,h\}$-derivation on $\mathbb{H}_{\mathbb{R}}$.

\begin{theorem}
\label{thm4}
Let $f$ be a Jordan left $\{g,h\}$-derivation over $\mathbb{H}_{\mathbb{R}}$. Then $f$ is a left $\{g,h\}$-derivation on $\mathbb{H}_{\mathbb{R}}$ if and only if $f(i)=f(j)=f(k)=0$.
\end{theorem}
\begin{proof}
Let $f(i)=f(j)=f(k)=0$ and $q=a+bi+cj+dk \in \mathbb{H}_{\mathbb{R}}$, where $a,b,c,d\in \mathbb{R}$. So, $f(q)=af(1)$.

Let $r=x+yi+zj+tk\in \mathbb{H}_{\mathbb{R}}$. We have to prove $f(qr)=qg(r)+rh(q)=qh(r)+rg(q)$. Since
\begin{align*}
&2(qg(r)+rh(q))=f(qr+rq)=2(ax-by-cz-dt)f(1) \\
&\implies qg(r)+rh(q)=(ax-by-cz-dt)f(1)=f(qr).
\end{align*}
Similarly, $qh(r)+rg(q)=f(qr)$. Therefore, $f$ is a left $\{g,h\}$-derivation on $\mathbb{H}_{\mathbb{R}}$.\\
The converse is true by the proof of Proposition \ref{thm3}.
\end{proof}

Now, we characterize Jordan left $\{g,h\}$-derivation on $\mathbb{H}_{\mathbb{R}}$.

\begin{theorem}
Let $\mathbb{H}_{\mathbb{R}}$ be the quaternion algebra over the field of real numbers. Then $f$ is a Jordan left $\{g, h\}$-derivation over $\mathbb{H}_{\mathbb{R}}$ if and only if $g=h$ and there exists $4$ elements in $\mathbb{R}$ such that
\begin{align*}
&g(q)=(aa_g^{(1)}-bb_g^{(1)}-cc_g^{(1)}-dd_g^{(1)})+(ab_g^{(1)}+ba_g^{(1)}+cd_g^{(1)}-dc_g^{(1)})i\\
&+(ac_g^{(1)}-bd_g^{(1)}+ca_g^{(1)}+db_g^{(1)})j
+(ad_g^{(1)}+bc_g^{(1)}-cb_g^{(1)}+da_g^{(1)})k,\\
&f(q)=2g(q),~\text{for all}~q\in \mathbb{H}_{\mathbb{R}}\\
&\text{where}~ q=a+bi+cj+dk~ \text{and}~ a_g^{(1)},b_g^{(1)},c_g^{(1)},d_g^{(1)}\in \mathbb{R}.
\end{align*}

Moreover, if $f$ is a Jordan left $\{g, g\}$-derivation over $\mathbb{H}_{\mathbb{R}}$, then $f$ and $g$ are right centralizers.
\end{theorem}

\begin{proof}
Let
\begin{equation}
\label{m1}
\begin{aligned}
F(l)=a_{F}^{(l)}+b_{F}^{(l)}i+c_{F}^{(l)}j+d_{F}^{(l)}k, ~\text{for}~ F=f, g, h~\text{and}~l=1,i,j,k.
\end{aligned}
\end{equation}

Since $f(1)=g(1)+h(1)$,
\begin{equation}
\label{m2}
\begin{aligned}
a_{f}^{(1)}=a_{g}^{(1)}+a_{h}^{(1)},~ b_{f}^{(1)}=b_{g}^{(1)}+b_{h}^{(1)},~ c_{f}^{(1)}=c_{g}^{(1)}+c_{h}^{(1)}, ~ d_{f}^{(1)}=d_{g}^{(1)}+d_{h}^{(1)}.
\end{aligned}
\end{equation}

Since, $i^2=j^2=k^2=-1$,

\begin{equation}
\label{m3}
\begin{aligned}
& a_{f}^{(1)}=b_{g}^{(i)}+b_{h}^{(i)},~ b_{f}^{(1)}=-a_{g}^{(i)}-a_{h}^{(i)},~ c_{f}^{(1)}=d_{g}^{(i)}+d_{h}^{(i)}, ~d_{f}^{(1)}=-c_{g}^{(i)}-c_{h}^{(i)},\\
& a_{f}^{(1)}=c_{g}^{(j)}+c_{h}^{(j)},~ b_{f}^{(1)}=-d_{g}^{(j)}-d_{h}^{(j)},~ c_{f}^{(1)}=-a_{g}^{(j)}-a_{h}^{(j)},~ d_{f}^{(1)}=b_{g}^{(j)}+b_{h}^{(j)},\\
& a_{f}^{(1)}=d_{g}^{(k)}+d_{h}^{(k)},~ b_{f}^{(1)}=c_{g}^{(k)}+c_{h}^{(k)},~ c_{f}^{(1)}=-b_{g}^{(k)}-b_{h}^{(k)},~ d_{f}^{(1)}=-a_{g}^{(k)}-a_{h}^{(k)}.
\end{aligned}
\end{equation}

Since $2i=1\circ i=i\circ 1$, using \eqref{m1},
\begin{equation}
\label{m4}
\begin{aligned}
& a_{f}^{(i)}=a_{h}^{(i)}-b_{g}^{(1)}=a_{g}^{(i)}-b_{h}^{(1)},~b_{f}^{(i)}=b_{h}^{(i)}+a_{g}^{(1)}=b_{g}^{(i)}+a_{h}^{(1)}\\
& c_{f}^{(i)}=c_{h}^{(i)}-d_{g}^{(1)}=c_{g}^{(i)}-d_{h}^{(1)},~d_{f}^{(i)}=d_{h}^{(i)}+c_{g}^{(1)}=d_{g}^{(i)}+c_{h}^{(1)}.
\end{aligned}
\end{equation}
By using \eqref{m2}-\eqref{m4},
\begin{equation}
\label{m5}
\begin{aligned}
& 2a_{f}^{(i)}=a_{h}^{(i)}-b_{g}^{(1)}+a_{g}^{(i)}-b_{h}^{(1)}=(a_{g}^{(i)}+a_{h}^{(i)})-(b_{g}^{(1)}+b_{h}^{(1)})=-2b_{f}^{(1)}\\
& \implies a_{f}^{(i)}=-b_{f}^{(1)}.~
 \text{Similarly,}~
 b_{f}^{(i)}=a_{f}^{(1)},~
 c_{f}^{(i)}=-d_{f}^{(1)},~
 d_{f}^{(i)}=c_{f}^{(1)}.
\end{aligned}
\end{equation}
Similarly, using $2j=1\circ j=j\circ 1$, $2k=1\circ k=k\circ 1$, \eqref{m3} and \eqref{m4}, we have
\begin{equation}
\label{m6}
\begin{aligned}
 & a_{f}^{(j)}=-c_{f}^{(1)},~
 b_{f}^{(j)}=d_{f}^{(1)},~
 c_{f}^{(j)}=a_{f}^{(1)},~
 d_{f}^{(j)}=-b_{f}^{(1)},~\\
 & a_{f}^{(k)}=-d_{f}^{(1)},~
 b_{f}^{(k)}=-c_{f}^{(1)},~
 c_{f}^{(k)}=b_{f}^{(1)},~
 d_{f}^{(k)}=a_{f}^{(1)}.
\end{aligned}
\end{equation}
Also, in view of \eqref{m3}-\eqref{m6}, for $F=g$ and $h$,
\begin{equation}
\label{m7}
\begin{aligned}
 & a_{F}^{(i)}=-b_{F}^{(1)},~
 b_{F}^{(i)}=a_{F}^{(1)},~
 c_{F}^{(i)}=-d_{F}^{(1)},~
 d_{F}^{(i)}=c_{F}^{(1)},\\
 & a_{F}^{(j)}=-c_{F}^{(1)},~
 b_{F}^{(j)}=d_{F}^{(1)},~
 c_{F}^{(j)}=a_{F}^{(1)},~
 d_{F}^{(j)}=-b_{F}^{(1)},~\\
 & a_{F}^{(k)}=-d_{F}^{(1)},~
 b_{F}^{(k)}=-c_{F}^{(1)},~
 c_{F}^{(k)}=b_{F}^{(1)},~
 d_{F}^{(k)}=a_{F}^{(1)}.
\end{aligned}
\end{equation}

Since $i\circ j=0$,
\begin{equation}
\label{m8}
\begin{aligned}
b_{g}^{(j)}=-c_{h}^{(i)},~ a_{g}^{(j)}=-d_{h}^{(i)},~d_{g}^{(j)}=a_{h}^{(i)},~ c_{g}^{(j)}=b_{h}^{(i)}.
\end{aligned}
\end{equation}
Again, by \eqref{m7} and \eqref{m8}, $x_{g}^{(1)}=x_{h}^{(1)}$, for all $x=a,b,c,d$. Therefore, $g=h$.
Let $q\in\mathbb{H}_{\mathbb{R}}$. Then $q=a+bi+cj+dk$,  where $a,b,c,d\in \mathbb{R}$.
Hence, by \eqref{m2} and \eqref{m5}-\eqref{m8},
\begin{equation}
\label{m9}
\begin{aligned}
&g(q)=(aa_g^{(1)}-bb_g^{(1)}-cc_g^{(1)}-dd_g^{(1)})+(ab_g^{(1)}+ba_g^{(1)}+cd_g^{(1)}-dc_g^{(1)})i\\
&+(ac_g^{(1)}-bd_g^{(1)}+ca_g^{(1)}+db_g^{(1)})j
+(ad_g^{(1)}+bc_g^{(1)}-cb_g^{(1)}+da_g^{(1)})k,\\
&f(q)=2g(q).
\end{aligned}
\end{equation}
Moreover, the required number of elements from $\mathbb{R}$ to express $g$ and $f$ is $4$.\\
Conversely, let $g$ and $f$ be of the form \eqref{m9}, for all $q\in\mathbb{H}_{\mathbb{R}}$, where $q=a+bi+cj+dk$ and $a,b,c,d\in \mathbb{R}$. Then $g(q)=q\alpha$ and $f(q)=2q\alpha$, where $\alpha=a_{g}^{(1)}+b_{g}^{(1)}i+c_{g}^{(1)}j+d_{g}^{(1)}k$. Now, let $p\in\mathbb{H}_{\mathbb{R}}$. Then $p=x+yi+zj+tk$, where $x,y,z,t\in \mathbb{R}$. Also, by direct computation,
\begin{align*}
f(pq+qp)=2(pg(q)+qg(p)).
\end{align*}
The last part can easily be derived from the converse part.
\end{proof}
 Below is an example of a Jordan left $\{g,g\}$-derivation on $\mathbb{H}_{\mathbb{R}}$ which is not a left $\{g,g\}$-derivation. Also, the condition of Theorem \ref{thm4} is not satisfied in this case.
\begin{example}
Let $q=a+bi+cj+dk \in \mathbb{H}_{\mathbb{R}}$, where $a,b,c,d\in \mathbb{R}$. Define $f,g:\mathbb{H}_{\mathbb{R}}\rightarrow \mathbb{H}_{\mathbb{R}}$ as $f(q)=2q$ and $g(q)=q$ respectively. Then $f,g$ are linear maps and $f$ is also a Jordan left $\{g,g\}$-derivation. But $f(ij)=f(k)=2k\neq 0=ij+ji=ig(j)+jg(i)$, so $f$ is not a left $\{g,g\}$-derivation.
\end{example}
Our next result shows that there is no nonzero left $\{g,g\}$-derivation on $\mathbb{H}_{\mathbb{R}}$.
\begin{theorem}
\label{thm7}
Let $f$ be a left $\{g,g\}$-derivation over $\mathbb{H}_{\mathbb{R}}$. Then $g$ and $f$ are identically zero.
\end{theorem}

\begin{proof}
Let $g$ and $f$ be of the form \eqref{m1}. Since every left $\{g, g\}$-derivation over $\mathbb{H}_{\mathbb{R}}$ is a Jordan left $\{g, g\}$-derivation, so \eqref{m7} and \eqref{m9} hold. By Theorem \ref{thm4}, $f(i)=f(j)=f(k)=0$. Since $i=1i=i1$, using \eqref{m1},
\begin{equation}
\label{m10}
\begin{aligned}
a_{g}^{(i)}=b_{g}^{(1)},~b_{g}^{(i)}=-a_{g}^{(1)},~ c_{g}^{(i)}=d_{g}^{(1)},~d_{g}^{(i)}=-c_{g}^{(1)}.
\end{aligned}
\end{equation}

Since $ij=k$ and $f(k)=0$,
\begin{equation}
\label{m11}
\begin{aligned}
b_{g}^{(j)}=-c_{g}^{(i)},~ a_{g}^{(j)}=-d_{g}^{(i)},~d_{g}^{(j)}=a_{g}^{(i)},~ c_{g}^{(j)}=b_{g}^{(i)}.
\end{aligned}
\end{equation}

Now, by \eqref{m7}, \eqref{m10} and \eqref{m11}, $a_{g}^{(1)}=b_{g}^{(1)}=c_{g}^{(1)}=d_{g}^{(1)}=0$. So, by \eqref{m9}, $g=0=f$.
\end{proof}

Finally, we give an example of a Jordan left $\{g, g\}$-derivation $f$ on $\mathbb{H}_{\mathbb{R}}$ where $g$ and $f$ are not left centralizers.
\begin{example}
Suppose $q=a+bi+cj+dk \in \mathbb{H}_{\mathbb{R}}$ and $g,f:\mathbb{H}_{\mathbb{R}} \rightarrow \mathbb{H}_{\mathbb{R}}$ are defined by $g(q)=(a-2b-3c-4d)+(2a+b+4c-3d)i
+(3a-4b+c+2d)j+(4a+3b-2c+d)k$ and $f(q)= 2g(q)$
respectively. Then $f$ is a Jordan left $\{g, g\}$
derivation on $\mathbb{H}_{\mathbb{R}}$. Further, if $p=5+6i+7j+8k$ and $q=9+10i+11j+12k$, then it can be seen that $g(pq)\neq g(p)q$ and $f(pq)\neq f(p)q$. Therefore, $g$ and $f$ are not left centralizers.
\end{example}

\begin{remark}
By Frobenious theorem, every finite dimensional noncommutative division algebra over $\mathbb{R}$ is isomorphic to $\mathbb{H}_{\mathbb{R}}$. Therefore, Theorem \ref{thm7} is true for every noncommutative finite dimensional division algebra over $\mathbb{R}$.
\end{remark}

\section*{Acknowledgement}
The authors are thankful to DST, Govt. of India for financial support and Indian Institute of Technology Patna for providing the research facilities.


\begin{thebibliography}{00}
\bibitem {ash} M. Ashraf and N. U. Rehmann,  On Lie ideals and Jordan left derivations of prime rings, \textit{Arch. Math. (Brno)} \textbf{36}(3) (2000) 201-206.

\bibitem {ben} D. Benkovi\v{c}, Jordan derivations and antiderivations on triangular matrices, \textit{Linear Algebra Appl.}
\textbf{397} (2005) 235-244.

\bibitem {bre} M. Bre\v{s}ar, Jordan derivations on semiprime rings, \textit{Proc. Amer. Math. Soc.} \textbf{104} (1988) 1003-1006.

\bibitem{bres} M. Bre\v{s}ar and J. Vukman, On left derivations and related mappings, \textit{Proc. Amer. Math. Soc.} \textbf{110}(1) (1990) 7-16.

\bibitem {bresa} M. Bre\v{s}ar, Jordan derivations revisited, \textit{Math. Proc. Cambridge Philos. Soc.} \textbf{139}(3) (2005) 411-425.

\bibitem{bresar} M. Bre\v{s}ar, Jordan \{g, h\}-derivations on tensor products of algebras, \textit{Linear Multilinear Algebra} \textbf{64}(11) (2016), 2199-2207.

\bibitem {cus} J.M. Cusack, Jordan derivations on rings, \textit{Proc. Amer. Math. Soc.} \textbf{53}(2) (1975) 321-324.

\bibitem{den} Q. Deng, On Jordan left derivations, \textit{Math. J. Okayama Univ.} \textbf{34} (1992) 145-147.

\bibitem {gho} N. M. Ghosseiri, Jordan derivations of some classes of matrix rings, \textit{Taiwanese J. Math.} \textbf{11}(1) (2007), 51-62.

\bibitem{ghos} N. M. Ghosseiri, On Jordan left derivations and generalized Jordan left derivations of matrix rings, \textit{Bull. Iranian Math. Soc.} \textbf{38}(3) (2012), 689-698.

\bibitem {her} I. N. Herstein, Jordan derivations of prime rings, \textit{Proc. Amer. Math. Soc.} \textbf{8} (1957) 1104-1110.

\bibitem{vuk} J. Vukman, On left Jordan derivations of rings and Banach algebras, \textit{Aequationes Math.} \textbf{75}(3) (2008) 260-266.

\bibitem{xu} X. W. Xu  and H. Y. Zhang, Jordan left derivations in full and upper triangular matrix rings, \textit{Electron. J. Linear Algebra} \textbf{20} (2010) 753-759.

\bibitem {zha} J. Zhang, Jordan derivations of nest algebras, \textit{Acta Math. Sinica} \textbf{41} (1998) 205-212.

\bibitem {zhan} J.-H. Zhang and W.-Y. Yu, Jordan derivations of triangular algebras, \textit{Linear Algebra Appl.} \textbf{419}(1) (2006) 251-255.









\end{thebibliography}
\end{document}